\documentclass[12pt,reqno]{amsart}
\usepackage{amssymb}
\usepackage{amsfonts}
\usepackage{hyperref}

\theoremstyle{plain}

\newtheorem{algorithm}{Algorithm}[section]
\newtheorem{thm}{Thm}

\newtheorem{lemma}[algorithm]{Lemma}

\newtheorem{theoremlet}[thm]{Theorem}

\newtheorem{lemmalet}[thm]{Lemma}

\newtheorem*{remarknonum}{Remark}

\newtheorem*{definitionnonum}{Definition}

\newtheorem{keylemma}[algorithm]{Key Lemma}

\numberwithin{equation}{algorithm}

\input{tcilatex.tex}
\input{tcilatex}
\begin{document}
\title{Almost Nonnegative Curvature on Some Fake $\mathbb{R}P^{6}$s and $%
\mathbb{R}P^{14}$s}
\author{Priyanka Rajan}
\address{Department of Mathematics\\
University of California\\
Riverside, CA 92521}
\email{rajan@math.ucr.edu}
\urladdr{https://sites.google.com/site/priyankarrajangeometry/home}
\author{Frederick Wilhelm}
\thanks{This work was supported by a grant from the Simons Foundation
(\#358068, Frederick Wilhelm)}
\address{Department of Mathematics\\
University of California\\
Riverside, CA 92521}
\email{fred@math.ucr.edu}
\urladdr{https://sites.google.com/site/frederickhwilhelmjr/home}
\date{October 15, 2015}

\begin{abstract}
We apply the lifting theorem of Searle and the second author to put metrics
of almost nonnegative curvature on the fake $\mathbb{R}P^{6}$'s of Hirsch
and Milnor and on the analogous fake $\mathbb{R}P^{14}$'s.
\end{abstract}

\maketitle

One of the great unsolved problems of Riemannian geometry is to determine
the structure of collapse with a lower curvature bound. An apparently
simpler, but still intractable problem, is to determine which closed
manifolds collapse to a point with a lower curvature bound. Such manifolds
are called almost nonnegatively curved. Here we construct almost nonnegative
curvature on some fake $\mathbb{R}P^{6}$s and $\mathbb{R}P^{14}$s.

\begin{theoremlet}
\label{main thm}The Hirsch-Milnor fake $\mathbb{R}P^{6}$s and the analogous
fake $\mathbb{R}P^{14}$s admit Riemannian metrics that simultaneously have
almost nonnegative sectional curvature and positive Ricci curvature.
\end{theoremlet}

\begin{remarknonum}
By considering cohomogeneity one actions on Brieskorn varieties, Schwachh%
\"{o}fer and Tuschmann observed in \cite{SchTu1} that in each odd dimension
of the form, $4k+1,$ there are at least $4^{k}$ oriented diffeomorphism
types of homotopy $\mathbb{R}P^{4k+1}$s that admit metrics that
simultaneously have positive Ricci curvature and almost nonnegative
sectional curvature.
\end{remarknonum}

The Hirsch-Milnor fake $\mathbb{R}P^{6}$s are quotients of free involutions
on the images of embeddings $\iota $ of the standard $6$--sphere, $\mathbb{S}%
^{6},$ into some of the Milnor exotic $7$--spheres, $\Sigma _{k}^{7}$ (\cite%
{HircshM}, \cite{Mil}). Our proof begins with the observation that the $%
SO\left( 3\right) $--actions that Davis constructed on the $\Sigma _{k}^{7}$%
s in \cite{Dav} leave these Hirsch-Milnor $S^{6}$s invariant and commute
with the Hirsch-Milnor free involution. Next we compare the
Hirsch-Milnor/Davis $\left( SO\left( 3\right) \times \mathbb{Z}_{2}\right) $%
--action on $\iota \left( S^{6}\right) \subset \Sigma _{k}^{7}$ with a very
similar linear action of $\left( SO\left( 3\right) \times \mathbb{Z}%
_{2}\right) $ on $\mathbb{S}^{6}$ $\subset \mathbb{R}^{7}$ and apply the
following lifting result of Searle and the second author.

\begin{theoremlet}
\label{lifting thm}\emph{(See Proposition 8.1 and Theorems B and C in \cite%
{SearW})} Let $\left( M_{e},G\right) $ and $\left( M_{s},G\right) $ be
smooth, compact, $n$--dimensional $G$--manifolds with $G$ a compact Lie
group. Suppose that the orbit spaces $M_{e}/G$ and $M_{s}/G$ are equivalent,
and $M_{s}/G$ has almost nonnegative curvature. Then $M_{e}$ admits a $G$%
--invariant family of metrics that has almost nonnegative sectional
curvature. Moreover, if the principal orbits of $\left( M_{e},G\right) $
have finite fundamental group and the quotient of the principal orbits of $%
M_{s}$ has Ricci curvature $\geq 1$, then every metric in the almost
nonnegatively curved family on $M_{e}$ can be chosen to also have positive
Ricci curvature.
\end{theoremlet}

We emphasize that to apply Theorem \ref{lifting thm}, $M_{s}/G$ need not be
a Riemannian manifold, but since $M_{s}$ is compact, $M_{s}/G$ is an
Alexandrov space with curvature bounded from below. The meaning of almost
nonnegative curvature for Alexandrov spaces is as follows.

\begin{definitionnonum}
We say that a sequence of Alexandrov spaces $\left\{ \left( X,\mathrm{dist}%
_{\alpha }\right) \right\} _{\alpha }$ is almost nonnegatively curved if and
only if there is a $D>0$ so that 
\begin{equation*}
\mathrm{sec}\left( X,g_{\alpha }\right) \geq -\frac{1}{\alpha }\text{ and }%
\mathrm{Diam}\left( X,g_{\alpha }\right) \leq D,
\end{equation*}%
or equivalently, after a rescaling, $X$ collapses to a point with a uniform
lower curvature bound.
\end{definitionnonum}

The following is the precise notion of equivalence of orbit spaces required
by the hypotheses of Theorem \ref{lifting thm}.

\begin{definitionnonum}
Suppose $G$ acts on $M_{e}$ and on $M_{s}$. We say that the orbit spaces $%
M_{e}/G$ and $M_{s}/G$ are equivalent if and only if there is a
strata-preserving homeomorphism $\Phi :M_{e}/G\longrightarrow M_{s}/G$ whose
restriction to each stratum is a diffeomorphism with the following property:

Let $\pi _{s}:M_{s}\longrightarrow M_{s}/G$ and $\pi
_{e}:M_{e}\longrightarrow M_{s}/G$ be the quotient maps. If $\mathcal{S}%
\subset M_{e}$ is a stratum, then for any $x_{e}\in \mathcal{S}$ and any $%
x_{s}\in \pi _{s}^{-1}\left( \Phi \left( \pi _{e}\left( x_{e}\right) \right)
\right) ,$ the action of $G_{x_{e}}$ on $\nu \left( \mathcal{S}\right)
_{x_{e}}$ is linearly equivalent to the action of $G_{x_{s}}$ on $\nu \left( 
\mathcal{S}\right) _{x_{s}}$. Here $G_{x}$ is the isotropy subgroup at $x$
and $\nu \left( \mathcal{S}\right) _{x}$ is the normal space to $\mathcal{S}$
at $x.$
\end{definitionnonum}

To construct the metrics on the fake $\mathbb{R}P^{6}$s of Theorem \ref{main
thm}, we apply Theorem \ref{lifting thm} with $G=\left( SO\left( 3\right)
\times \mathbb{Z}_{2}\right) .$ $M_{e}$ will be the Hirsch-Milnor embedded
image of $\mathbb{S}^{6}$ in $\Sigma _{k}^{7},$ and $M_{s}$ will be $\mathbb{%
S}^{6}$ with the following $\left( SO\left( 3\right) \times \mathbb{Z}%
_{2}\right) $--action: View $\mathbb{S}^{6}$ as the unit sphere in $\mathbb{%
H\oplus }\func{Im}\mathbb{H},$ where $\mathbb{H}$ stands for the
quaternions, and let $SO\left( 3\right) \times \mathbb{Z}_{2}$ act on $%
\mathbb{S}^{6}\subset \mathbb{H\oplus }\func{Im}\mathbb{H}$ via 
\begin{eqnarray}
SO\left( 3\right) \times \mathbb{Z}_{2}\times \mathbb{S}^{6}
&\longrightarrow &\mathbb{S}^{6}  \notag \\
\left( g,\pm ,\left( a,c\right) \right) &\mapsto &\pm \left( g\left(
a\right) ,g\left( c\right) \right) .  \label{quat lin act}
\end{eqnarray}%
Here the $SO\left( 3\right) $--action on the $\mathbb{H}$--factor is the
direct sum of the standard action of $SO\left( 3\right) $ on $\func{Im}%
\mathbb{H}$ with the trivial action on $\func{Re}\left( \mathbb{H}\right) .$

Since quotient maps of isometric group actions preserve lower curvature
bounds,\linebreak\ $\mathbb{S}^{6}/\left( SO\left( 3\right) \times \mathbb{Z}%
_{2}\right) $ has curvature $\geq 1$ (\cite{BGP})$.$ Thus to construct the
metrics on the fake $\mathbb{R}P^{6}$s of Theorem \ref{main thm}, it
suffices to combine Theorem \ref{lifting thm} with the following result.

\begin{lemmalet}
\label{quat lem}The orbit space of the Hirsch-Milnor/Davis action of $%
SO\left( 3\right) \times \mathbb{Z}_{2}$ on $\iota \left( \mathbb{S}%
^{6}\right) \subset \Sigma _{k}^{7}$ is equivalent to the orbit space of the
linear action (\ref{quat lin act}) on $\mathbb{S}^{6}.$
\end{lemmalet}

Our metrics on fake $\mathbb{R}P^{14}$s are octonionic analogs of our
metrics on fake $\mathbb{R}P^{6}$s. The analogy begins with Shimada's
observation that Milnor's proof of the total spaces of certain $\mathbb{S}%
^{3}$--bundles over $\mathbb{S}^{4}$ being exotic spheres also applies to
certain $\mathbb{S}^{7}$--bundles over $\mathbb{S}^{8}$ (\cite{Shim}).
Davis's construction of the $SO\left( 3\right) $--actions on $\Sigma
_{k}^{7} $s is based on the fact that $SO\left( 3\right) $ is the group of
automorphisms of $\mathbb{H}.$ Exploiting the fact that $G_{2}$ is the group
automorphisms of the octonions, $\mathbb{O},$ Davis constructs analogous $%
G_{2}$ actions on Shimada's exotic $\Sigma _{k}^{15}$s. By applying a result
of Brumfiel (\cite{Bruf}), we will see that the Hirsch and Milnor
construction of fake $\mathbb{R}P^{6}$s as quotients of $\iota \left( 
\mathbb{S}^{6}\right) \subset \Sigma _{k}^{7}$ also works to construct fake $%
\mathbb{R}P^{14}$s as quotients of $\iota \left( \mathbb{S}^{14}\right)
\subset \Sigma _{k}^{15}.$ Thus to construct the fake $\mathbb{R}P^{14}$s of
Theorem \ref{main thm}, it suffices to show the following.

\begin{lemmalet}
\label{oct lemma}The orbit space of the Hirsch-Milnor/Davis action of $%
G_{2}\times \mathbb{Z}_{2}$ on $\iota \left( \mathbb{S}^{14}\right) \subset
\Sigma _{k}^{15}$ is equivalent to the orbit space of the following linear
action of $G_{2}\times \mathbb{Z}_{2}$ on $\mathbb{S}^{14}\subset \mathbb{O}%
\oplus \func{Im}\mathbb{O},$%
\begin{eqnarray}
G_{2}\times \mathbb{Z}_{2}\times \mathbb{S}^{14} &\longrightarrow &\mathbb{S}%
^{14}  \notag \\
\left( g,\pm ,\left( a,c\right) \right) &\mapsto &\pm \left( g\left(
a\right) ,g\left( c\right) \right) .  \label{oct lin act}
\end{eqnarray}
\end{lemmalet}

In Section 1, we review the construction of the Hirsch-Milnor and Davis
actions and explain why the Hirsch-Milnor construction works in the
Octonionic case. In Section 2, we prove Lemmas \ref{quat lem} and \ref{oct
lemma} and hence Theorem \ref{main thm}, and in Section 3, we make some
concluding remarks. We refer the reader to page 185 of \cite{Baez} for a
description of how $G_{2}$ acts as automorphisms of the Octonions.

\begin{remarknonum}
Explicit formulas for exotic involutions on $\mathbb{S}^{6}$ and $\mathbb{S}%
^{14}$ are given in (\cite{ADPR}), where it is shown, on pages 13--17, that
the corresponding fake $\mathbb{R}P^{6}$ is diffeomorphic to the
Hirsch--Milnor $\mathbb{R}P^{6}$ that corresponds to $\Sigma _{3}^{7}$.
\end{remarknonum}

\noindent \textbf{Acknowledgement. }\emph{It is a pleasure to thank Reinhard
Schultz for stimulating conversations with the first author about
topological aspects of this paper. We are grateful to a referee for pointing
out the Schwachh\"{o}fer and Tuschmann result in \cite{SchTu1} about
homotopy }$RP^{4k+1}$\emph{s admitting metrics that simultaneously have
positive Ricci curvature and almost nonnegative sectional curvature.}

\section{How to Construct Exotic Real Projective Spaces}

In this section, we review Milnor spheres, the Hirsch-Milnor construction,
and the Davis actions. We then explain how the Hirsch-Milnor argument gives
fake $\mathbb{R}P^{14}$s.

To construct the Milnor spheres, we write $\Lambda $ for $\mathbb{H}$ or $%
\mathbb{O}$ and $b$ for the real dimension of $\Lambda .$ To get a $\mathbb{S%
}^{b-1}$--bundle over $\mathbb{S}^{b}$ with structure group $SO\left(
b\right) ,$ $\left( E_{h,j},p_{h,j}\right) ,$ we glue two copies of $\Lambda
\times \mathbb{S}^{b-1}$ together via 
\begin{eqnarray}
\Phi _{h,j} &:&\Lambda \setminus \left\{ 0\right\} \times \mathbb{S}%
^{b-1}\longrightarrow \Lambda \setminus \left\{ 0\right\} \times \mathbb{S}%
^{b-1}  \notag \\
\Phi _{h,j} &:&\left( u,q\right) \mapsto \left( \frac{u}{\left\vert
u\right\vert ^{2}},\left( \frac{u}{\left\vert u\right\vert }\right)
^{h}q\left( \frac{u}{\left\vert u\right\vert }\right) ^{j}\right) .
\label{gluing eqn}
\end{eqnarray}

To define the projection $p_{h,j}:$ $E_{h,j}\longrightarrow \mathbb{S}^{b},$
we think of $\mathbb{S}^{b}$ as obtained by gluing together two copies of $%
\Lambda $ along $\Lambda \setminus \left\{ 0\right\} $ via $u\mapsto \frac{u%
}{\left\vert u\right\vert ^{2}}.$ $p_{h,j}$ is then defined to be the
projection to either copy of $\Lambda .$

When $h+j=\pm 1$, the smooth function%
\begin{equation*}
f:(u,q)\mapsto \frac{\func{Re}(q)}{\sqrt{1+|u|^{2}}}=\frac{\func{Re}\left(
vr^{-1}\right) }{\sqrt{1+\left\vert v\right\vert ^{2}}}
\end{equation*}%
is regular except at $\left( u,q\right) =\left( 0,\pm 1\right) .$ Hence, $%
E_{h,j}$ is homeomorphic to $\mathbb{S}^{2b-1}$ if $h+j=\pm 1,$ and a
Mayer-Vietoris argument shows that $E_{h,j}$ is not homeomorphic to $\mathbb{%
S}^{2b-1}$ if $h+j\neq \pm 1.$ Since $f\left( 0,\pm 1\right) =\pm 1,$ it
also follows that $f^{-1}\left( 0\right) $ is diffeomorphic to $\mathbb{S}%
^{2b-2}.$

From now on we assume that 
\begin{equation}
h+j=1,  \label{sphere eqn}
\end{equation}%
and we set 
\begin{equation}
k=h-j.  \label{dfn of k}
\end{equation}%
So%
\begin{equation*}
k=2h-1.
\end{equation*}%
For simplicity, we will write $\Sigma _{k}^{2b-1}$ for $E_{h,j}$ and $\Phi
_{k}$ for $\Phi _{h,j},$ and set 
\begin{equation*}
S_{k}^{2b-2}\equiv f^{-1}\left( 0\right) .
\end{equation*}

The Hirsch-Milnor construction (\cite{HircshM}) begins with the observation
that the involution 
\begin{eqnarray*}
T &:&\Lambda \times \mathbb{S}^{b-1}\longrightarrow \Lambda \times \mathbb{S}%
^{b-1} \\
T &:&\left( u,q\right) \mapsto \left( u,-q\right)
\end{eqnarray*}%
induces a well-defined free involution of $\Sigma _{k}^{2b-1}.$ Moreover, $T$
leaves $S_{k}^{2b-2}$ invariant.\label{dfn of T page} Lemma 3 of \cite%
{HircshM} says that the quotient of any fixed point free involution on $%
\mathbb{S}^{n}$ is homotopy equivalent to $\mathbb{R}P^{n}.$ In particular,
all of our spaces 
\begin{equation*}
P_{k}^{2b-2}\equiv S_{k}^{2b-2}/T
\end{equation*}%
are homotopy equivalent to $\mathbb{R}P^{2b-2}.$ Hirsch and Milnor then show
that when $b=4,$ $P_{k}^{6}$ is not diffeomorphic to $\mathbb{R}P^{6},$
provided $\Sigma _{k}^{7}$ is an odd element of $\Theta _{7},$ the group of
oriented diffeomorphism classes of differential structures on $\mathbb{S}%
^{7}.$ According to pages 102 and 103 of \cite{Ell-Kup}, there are $16$
oriented diffeomorphism classes among the $\Sigma _{k}^{7}$s, and among
these, $8$ are odd elements of $\Theta _{7}.$

To understand how this works octonionically, we let $\Theta _{15}$ be the
group of oriented diffeomorphism classes of differential structures on $%
\mathbb{S}^{15},$ and we let $bP_{16}$ be the set of the elements of $\Theta
_{15}$ that bound parallelizable manifolds. According to \cite{KerMil}, $%
bP_{16}$ is a cyclic subgroup of $\Theta _{15}$ of order $8,128$ and index $%
2,$ and according to Theorem 1.3 in \cite{Bruf}, $\Theta _{15}$ is not
cyclic. Thus 
\begin{eqnarray*}
\Theta _{15} &\cong &bP_{16}\oplus \mathbb{Z}_{2} \\
&\cong &\mathbb{Z}_{8,128}\oplus \mathbb{Z}_{2}.
\end{eqnarray*}%
According to Wall (\cite{Wall}), a homotopy sphere bounds a parallelizable
manifold if and only if it bounds a $7$--connected manifold. In particular,
each of the $\Sigma _{k}^{15}$s is in $bP_{16}.$

According to pages 101---107 of \cite{Ell-Kup}, $\Sigma _{k}^{15}$
represents an odd element of $bP_{16}$ if and only if $\frac{h\left(
h-1\right) }{2}$ is odd, that is, $h$ is congruent to $2$ or $3$ mod $4.$

The Hirsch-Milnor argument, combined with the fact that $\Theta _{15}\cong
bP_{16}\oplus \mathbb{Z}_{2},$ implies $P_{k}^{14}$ is not diffeomorphic to $%
\mathbb{R}P^{14},$ if $\Sigma _{k}^{15}$ is an odd element of $bP_{16}.$

We let 
\begin{equation*}
G^{\Lambda }\equiv \left\{ 
\begin{array}{ll}
SO\left( 3\right) & \text{when }\Lambda =\mathbb{H} \\ 
G_{2} & \text{when }\Lambda =\mathbb{O}.%
\end{array}%
\right.
\end{equation*}%
Davis observed that since $G^{\Lambda }$ is the automorphism group of $%
\Lambda ,$ the diagonal action 
\begin{eqnarray}
G^{\Lambda }\times \Lambda \times \mathbb{S}^{b-1} &\longrightarrow &\Lambda
\times \mathbb{S}^{b-1}  \label{Davis} \\
g\left( u,v\right) &=&\left( g\left( u\right) ,g\left( v\right) \right) 
\notag
\end{eqnarray}%
induces a well-defined $G^{\Lambda }$--action on $\Sigma _{k}^{2b-1}$ (\cite%
{Dav}).

Next we observe that the Davis action leaves $S_{k}^{2b-2}=$ $f^{-1}\left(
0\right) $ invariant and commutes with $T,$ giving us the $SO\left( 3\right)
\times \mathbb{Z}_{2}$ actions of Lemma \ref{quat lem} and the $G_{2}\times 
\mathbb{Z}_{2}$ actions of Lemma \ref{oct lemma}.

\section{Identifying the Orbit Spaces}

In this section, we prove Lemmas \ref{quat lem} and \ref{oct lemma}
simultaneously and hence Theorem \ref{main thm}. In Lemma \ref{Q_s Lemma}
(below), we identify the quotient map for the standard $G^{\Lambda }$%
--action of $\mathbb{S}^{2b-2}.$ In Lemma \ref{Q_e lemma} (below), we
identify the quotient map for the Davis action on $S_{k}^{2b-2}.$ Then in
Key Lemma \ref{key lemme}, we show that the two $G^{\Lambda }$ quotients are
the same. It is then a simple matter to identify the two $G^{\Lambda }\times 
\mathbb{Z}_{2}$ quotient spaces with each other.

\begin{lemma}
\label{Q_s Lemma}Let $\mathbb{S}^{2b-2}$ be the unit sphere in $\Lambda
\oplus \func{Im}\left( \Lambda \right) ,$ and let $\left\langle
,\right\rangle $ be the real dot product.

The map 
\begin{eqnarray*}
Q_{s} &:&\mathbb{S}^{2b-2}\longrightarrow Q_{s}\left( \mathbb{S}%
^{2b-2}\right) \subsetneq \mathbb{R}^{3} \\
\left( 
\begin{array}{c}
a \\ 
c%
\end{array}%
\right) &\longmapsto &\left( \left\vert a\right\vert ,\text{ }\func{Re}\,a,%
\text{ }\left\langle \func{Im}\,a,\func{Im}\,c\right\rangle \right)
\end{eqnarray*}%
has the following properties.

\noindent 1. The fibers of $Q_{s}$ coincide with the orbits of the $%
G^{\Lambda }$ action 
\begin{eqnarray*}
G^{\Lambda }\times \mathbb{S}^{2b-2} &\longrightarrow &\mathbb{S}^{2b-2} \\
\left( g,\left( a,c\right) \right) &\mapsto &\left( g\left( a\right)
,g\left( c\right) \right) .
\end{eqnarray*}

\noindent 2. The image of $Q_{s}$ is $Q_{s}\left( \mathbb{S}^{2b-2}\right) =$
\ 
\begin{equation*}
\left\{ \left. \left( x,y,z\right) \text{ }\right\vert \text{ }x\in \left[
0,1\right] \text{ }y\in \left[ -x,x\right] ,\text{ }z\in \left[ -\sqrt{%
\left( x^{2}-y^{2}\right) \left( 1-x^{2}\right) },\sqrt{\left(
x^{2}-y^{2}\right) \left( 1-x^{2}\right) }\right] \right\} .
\end{equation*}%
\noindent 3. The principal orbits are mapped to the interior of $Q_{s}\left( 
\mathbb{S}^{2b-2}\right) .$\ The fixed points are mapped to $\left(
1,1,0\right) $\ and $\left( 1,-1,0\right) ,$\ and the other orbits are
mapped to $\partial Q_{s}\left( \mathbb{S}^{2b-2}\right) \setminus \left\{
\left( 1,1,0\right) ,\left( 1,-1,0\right) \right\} .$
\end{lemma}

\begin{proof}
Part 2 follows from the observations that%
\begin{eqnarray*}
\left\vert a\right\vert &\in &\left[ 0,1\right] , \\
\func{Re}\,a &\in &\left[ -\left\vert a\right\vert ,\left\vert a\right\vert %
\right] , \\
\left\langle \func{Im}\,a,\func{Im}\,c\right\rangle &\in &\left[ -\left\vert 
\func{Im}\left( a\right) \right\vert \left\vert \func{Im}\left( c\right)
\right\vert ,\text{ }\left\vert \func{Im}\left( a\right) \right\vert
\left\vert \func{Im}\left( c\right) \right\vert \right] ,
\end{eqnarray*}%
and 
\begin{equation*}
\left\vert \func{Im}\left( a\right) \right\vert \left\vert \func{Im}\left(
c\right) \right\vert \in \left[ 0,\text{ }\sqrt{\left( \left\vert
a\right\vert ^{2}-\func{Re}\left( a\right) ^{2}\right) \left( 1-\left\vert
a\right\vert ^{2}\right) }\right] .
\end{equation*}

Since the three quantities $\left\vert a\right\vert ,$ $\func{Re}\,a,$ $%
\left\langle \func{Im}\,a,\func{Im}\,c\right\rangle $ are invariant under $%
G^{\Lambda },$ each orbit of $G^{\Lambda }$ is contained in a fiber of $%
Q_{s}.$

Conversely, if $\left( a_{1},c_{1}\right) $ and $\left( a_{2},c_{2}\right) $
satisfy $Q_{s}\left( a_{1},c_{1}\right) =Q_{s}\left( a_{2},c_{2}\right) ,$
then 
\begin{eqnarray*}
\left\vert a_{1}\right\vert &=&\left\vert a_{2}\right\vert \\
\func{Re}\left( a_{1}\right) &=&\func{Re}\left( a_{2}\right) ,\text{ and } \\
\left\langle \func{Im}\,a_{1},\func{Im}\,c_{1}\right\rangle &=&\left\langle 
\func{Im}\,a_{2},\func{Im}\,c_{2}\right\rangle .
\end{eqnarray*}%
Together with $\func{Re}\left( c_{i}\right) =0$ and $\left\vert
a_{i}\right\vert ^{2}+\left\vert c_{i}\right\vert ^{2}=1,$ this gives 
\begin{eqnarray*}
\left\vert \func{Im}\left( a_{1}\right) \right\vert &=&\left\vert \func{Im}%
\left( a_{2}\right) \right\vert \\
\left\vert \func{Im}\left( c_{1}\right) \right\vert &=&\left\vert \func{Im}%
\left( c_{2}\right) \right\vert .
\end{eqnarray*}%
Since we also have $\left\langle \func{Im}\,a_{1},\func{Im}%
\,c_{1}\right\rangle =\left\langle \func{Im}\,a_{2},\func{Im}%
\,c_{2}\right\rangle ,$ it follows that an element of $G^{\Lambda }$ carries 
$\left( a_{1},c_{1}\right) $ to $\left( a_{2},c_{2}\right) .$ This completes
the proof of Part 1.

To prove Part $3,$ we first note that the orbit of $\left( a,c\right) $ is
not principal if and only if 
\begin{equation*}
\left\vert \left\langle \func{Im}\,a,\func{Im}\,c\right\rangle \right\vert
=\left\vert \func{Im}\left( a\right) \right\vert \left\vert \func{Im}\left(
c\right) \right\vert ,
\end{equation*}%
and this is equivalent to $Q_{s}\left( a,c\right) \in \partial Q_{s}\left(
a,c\right) .$ So the principal orbits are mapped onto the interior of $%
Q_{s}\left( \mathbb{S}^{2b-2}\right) .$

On the other hand, the fixed points are $\left( \pm 1,0\right) $ and $%
Q_{s}\left( \pm 1,0\right) =\left( 1,\pm 1,0\right) $ as claimed.
\end{proof}

Before proceeding, recall that we view 
\begin{equation*}
\Sigma _{k}^{2b-1}=\left( \Lambda \times \mathbb{S}^{b-1}\right) \cup _{\Phi
_{k}}\left( \Lambda \times \mathbb{S}^{b-1}\right) ,
\end{equation*}%
where $\Phi _{k}$ is determined by Equations ($\ref{gluing eqn}$)$,$ (\ref%
{sphere eqn}), and (\ref{dfn of k}). Combining this with the definition of $%
S_{k}^{2b-2},$ we have that 
\begin{equation*}
S_{k}^{2b-2}=U_{1}\cup _{\Phi _{k}}U_{2},
\end{equation*}%
where 
\begin{eqnarray*}
U_{1} &\equiv &\left\{ \left. \left( u,q\right) \in \Lambda \times \mathbb{S}%
^{b-1}\text{ }\right\vert \text{ }\func{Re}\left( q\right) =0\right\} \text{
and} \\
U_{2} &\equiv &\left\{ \left. \left( v,r\right) \in \Lambda \times \mathbb{S}%
^{b-1}\text{ }\right\vert \text{ }\func{Re}\left( vr^{-1}\right) =\func{Re}%
\bar{v}r=0\right\} .
\end{eqnarray*}

The quotient map of the $G^{\Lambda }$--action on $S_{k}^{2b-2}$ has the
following description.

\begin{lemma}
\label{Q_e lemma}Let $\phi :R^{n}\longrightarrow R$\ be given by, $\phi (v)=%
\frac{1}{\sqrt{1+\left\vert v\right\vert ^{2}}}.$

The map%
\begin{eqnarray*}
Q_{k} &:&S_{k}^{2b-2}\longrightarrow Q_{k}\left( S_{k}^{2b-2}\right)
\subsetneq \mathbb{R}^{3} \\
Q_{k}|_{U_{1}}\left( u,q\right) &=&\phi (u)\left( \left\vert u\right\vert ,%
\text{ }\func{Re}\,uq,\text{ }\phi \left( u\right) \left\langle \func{Im}uq,%
\func{Im}q\right\rangle \right) \\
Q_{k}|_{U_{2}}\left( v,r\right) &=&\phi (v)\left( \left\vert r\right\vert ,%
\text{ }\func{Re}r,\text{ }\phi (v)\left\langle \func{Im}\,r,\func{Im}\,\bar{%
v}r\right\rangle \right)
\end{eqnarray*}

is well-defined and has fibers that coincide with the orbits of $G^{\Lambda
}.$
\end{lemma}

\begin{proof}
To see that $Q_{k}$ is well-defined, we will show 
\begin{equation}
Q_{k}|_{U_{1}\setminus \left\{ 0\times \mathbb{S}^{b-1}\right\}
}=Q_{k}|_{U_{2}\setminus \left\{ 0\times \mathbb{S}^{b-1}\right\} }\circ
\Phi _{k}|_{U_{1}\setminus \left\{ 0\times \mathbb{S}^{b-1}\right\} }.
\label{well dfned
eqn}
\end{equation}%
Since 
\begin{equation*}
\Phi _{k}\left( u,q\right) =\left( \frac{u}{\left\vert u\right\vert ^{2}}%
,\left( \frac{u}{\left\vert u\right\vert }\right) ^{h}q\left( \frac{u}{%
\left\vert u\right\vert }\right) ^{-\left( h-1\right) }\right) ,
\end{equation*}%
where $k=2h-1$, the left hand side of Equation (\ref{well dfned eqn}) is%
\begin{equation*}
Q_{k}|_{U_{2}\setminus \left\{ 0\times \mathbb{S}^{b-1}\right\} }\circ \Phi
_{k}|_{U_{1}\setminus \left\{ 0\times \mathbb{S}^{b-1}\right\} }\left(
u,q\right) =Q_{k}\left( \frac{u}{\left\vert u\right\vert ^{2}},\frac{%
u^{h}qu^{-(h-1)}}{|u|}\right)
\end{equation*}%
\begin{equation}
=\phi \left( \frac{u}{\left\vert u\right\vert ^{2}}\right) \left( \left\vert 
\frac{u^{h}qu^{-(h-1)}}{|u|}\right\vert ,\text{ }\func{Re}\,\frac{%
u^{h}qu^{-(h-1)}}{|u|},\phi \left( \frac{u}{\left\vert u\right\vert ^{2}}%
\right) \left\langle \func{Im}\,\frac{u^{h}qu^{-(h-1)}}{|u|},\func{Im}\,%
\frac{\bar{u}}{\left\vert u\right\vert ^{2}}\frac{u^{h}qu^{-(h-1)}}{|u|}%
\right\rangle \right)  \label{RHS eqn111}
\end{equation}

To see that this is equal to $Q_{k}|_{U_{1}\setminus \left\{ 0\times \mathbb{%
S}^{b-1}\right\} }\left( u,q\right) $, we will simplify each coordinate
separately. Before doing so we point out that 
\begin{eqnarray}
\frac{1}{\left\vert u\right\vert }\phi (\frac{u}{\left\vert u\right\vert ^{2}%
}) &=&\frac{1}{\left\vert u\right\vert }\frac{1}{\sqrt{1+\frac{1}{\left\vert
u\right\vert ^{2}}}}  \notag \\
&=&\frac{1}{\sqrt{\left\vert u\right\vert ^{2}+1}}  \label{Phi eqn} \\
&=&\phi \left( u\right) .  \notag
\end{eqnarray}%
So the first coordinate of the right hand side of Equation (\ref{RHS eqn111}%
) is 
\begin{eqnarray}
\phi (\frac{u}{\left\vert u\right\vert ^{2}})\left\vert \frac{%
u^{h}qu^{-(h-1)}}{|u|}\right\vert &=&\phi \left( \frac{u}{\left\vert
u\right\vert ^{2}}\right)  \notag \\
&=&\left\vert u\right\vert \phi (u),  \label{1st coord eqn}
\end{eqnarray}%
and the second coordinate of the right hand side of Equation (\ref{RHS
eqn111}) is%
\begin{eqnarray*}
\text{ }\phi \left( \frac{u}{\left\vert u\right\vert ^{2}}\right) \func{Re}\,%
\frac{u^{h}qu^{-(h-1)}}{|u|} &=&\text{ }\phi \left( \frac{u}{\left\vert
u\right\vert ^{2}}\right) \func{Re}\left( \frac{uq}{\left\vert u\right\vert }%
\right) \\
&=&\text{ }\frac{1}{\left\vert u\right\vert }\phi \left( \frac{u}{\left\vert
u\right\vert ^{2}}\right) \func{Re}\left( uq\right) \\
&=&\phi \left( u\right) \func{Re}\left( uq\right) ,\text{ by Equation (\ref%
{Phi eqn}).}
\end{eqnarray*}%
Finally, we have that the third coordinate of the right hand side of
Equation (\ref{RHS eqn111}) is 
\begin{eqnarray*}
&&\phi \left( \frac{u}{\left\vert u\right\vert ^{2}}\right) ^{2}\left\langle 
\func{Im}\,\frac{u^{h}qu^{-(h-1)}}{|u|},\func{Im}\,\frac{\bar{u}}{\left\vert
u\right\vert ^{2}}\frac{u^{h}qu^{-(h-1)}}{|u|}\right\rangle \\
&=&\phi \left( \frac{u}{\left\vert u\right\vert ^{2}}\right)
^{2}\left\langle \func{Im}\,\frac{u^{h}qu^{-(h-1)}}{|u|},\func{Im}\,\frac{%
u^{h-1}qu^{-(h-1)}}{|u|}\right\rangle \\
&=&\phi \left( \frac{u}{\left\vert u\right\vert ^{2}}\right) ^{2}\frac{1}{%
\left\vert u\right\vert ^{2}}\left\langle \func{Im}\,u^{h-1}\left( uq\right)
u^{-(h-1)},\func{Im}\,u^{h-1}\left( q\right) u^{-(h-1)}\right\rangle \\
&=&\phi \left( u\right) ^{2}\left\langle \func{Im}\,uq,\func{Im}%
\,q\right\rangle ,\text{ by Equation (\ref{Phi eqn}).}
\end{eqnarray*}

Combining the previous three displays with Equation (\ref{RHS eqn111}) and
the definition of $Q_{k}|_{U_{1}},$ we see that $Q_{k}:S_{k}^{2b-2}%
\longrightarrow Q_{k}\left( S_{k}^{2b-2}\right) \subsetneq \mathbb{R}^{3}$
is well-defined.

To see that $Q_{k}|_{U_{1}}$ is constant on each orbit of $G^{\Lambda },$ we
use the fact that $G^{\Lambda }$ acts by isometries and commutes with
conjugation to get 
\begin{eqnarray*}
\func{Re}g\left( u\right) g\left( q\right) &=&\left\langle g\left( u\right) ,%
\overline{g\left( q\right) }\right\rangle \\
&=&\left\langle g\left( u\right) ,g\left( \bar{q}\right) \right\rangle \\
&=&\left\langle u,\bar{q}\right\rangle \\
&=&\func{Re}\left( uq\right) .
\end{eqnarray*}%
We also have%
\begin{eqnarray*}
&&\left\langle \func{Im}\left( g\left( u\right) g\left( q\right) \right) ,%
\func{Im}g\left( q\right) \right\rangle \\
&=&\left\langle \func{Re}\left( g\left( u\right) \right) \func{Im}g\left(
q\right) +\func{Re}\left( g\left( q\right) \right) \func{Im}g\left( u\right)
+\func{Im}g\left( u\right) \func{Im}g\left( q\right) ,\text{ }\func{Im}%
g\left( q\right) \right\rangle \\
&=&\left\langle \func{Re}\left( u\right) \func{Im}g\left( q\right) +\func{Re}%
\left( q\right) \func{Im}g\left( u\right) ,\text{ }\func{Im}g\left( q\right)
\right\rangle \\
&=&\left\langle g\left( \func{Re}\left( u\right) \func{Im}\left( q\right) +%
\func{Re}\left( q\right) \func{Im}\left( u\right) \right) ,\text{ }g\left( 
\func{Im}\left( q\right) \right) \right\rangle \\
&=&\left\langle \func{Re}\left( u\right) \func{Im}\left( q\right) +\func{Re}%
\left( q\right) \func{Im}\left( u\right) ,\func{Im}\left( q\right)
\right\rangle \\
&=&\left\langle \func{Re}\left( u\right) \func{Im}\left( q\right) +\func{Re}%
\left( q\right) \func{Im}\left( u\right) +\func{Im}\,uq\func{Im}q,\text{ }%
\func{Im}\left( q\right) \right\rangle \\
&=&\left\langle \func{Im}\left( uq\right) ,\text{ }\func{Im}q\right\rangle .
\end{eqnarray*}

Since $\left\vert g\left( u\right) \right\vert =\left\vert u\right\vert $
and $\phi \left( gu\right) =\phi \left( u\right) ,$ it follows that 
\begin{equation*}
\left. Q_{k}\right\vert _{U_{1}}\left( 
\begin{array}{c}
g\left( u\right) \\ 
g\left( q\right)%
\end{array}%
\right) =\left. Q_{k}\right\vert _{U_{1}}\left( 
\begin{array}{c}
u \\ 
q%
\end{array}%
\right) .
\end{equation*}

Combining this with%
\begin{eqnarray*}
\left. Q_{k}\right\vert _{U_{2}}g\left( 
\begin{array}{c}
0 \\ 
r%
\end{array}%
\right) &=&\left( 1,\func{Re}\left( r\right) ,0\right) \\
&=&\left. Q_{k}\right\vert _{U_{2}}\left( 
\begin{array}{c}
0 \\ 
r%
\end{array}%
\right) ,
\end{eqnarray*}%
it follows that $Q_{k}$ is constant on each orbit of $G^{\Lambda }.$

On the other hand, if 
\begin{equation*}
Q_{k}|_{U_{1}}\left( u_{1},q_{1}\right) =Q_{k}|_{U_{1}}\left(
u_{2},q_{2}\right) ,
\end{equation*}%
then%
\begin{eqnarray}
\phi \left( u_{1}\right) \left\vert u_{1}\right\vert &=&\phi \left(
u_{2}\right) \left\vert u_{2}\right\vert ,  \label{first coord eqn} \\
\phi \left( u_{1}\right) ^{2}\left\langle \func{Im}\left( u_{1}q_{1}\right)
,q_{1}\right\rangle &=&\phi \left( u_{2}\right) ^{2}\left\langle \func{Im}%
\left( u_{2}q_{2}\right) ,q_{2}\right\rangle ,\text{ and\label{3rd comp eqn}}
\\
\phi \left( u_{1}\right) \func{Re}\,u_{1}q_{1} &=&\phi \left( u_{2}\right) 
\func{Re}\,u_{2}q_{2}.  \label{2nd coord eqn}
\end{eqnarray}

Equation (\ref{first coord eqn}) implies that $\left\vert u_{1}\right\vert
=\left\vert u_{2}\right\vert $ and $\phi \left( u_{1}\right) =\phi \left(
u_{2}\right) .$ So 
\begin{eqnarray*}
\func{Re}\left( u_{1}\right) &=&\func{Re}\left( u_{1}\right) \left\langle
q_{1},q_{1}\right\rangle \\
&=&\left\langle \left( \func{Re}\left( u_{1}\right) +\func{Im}\left(
u_{1}\right) \right) q_{1},q_{1}\right\rangle \text{, since }\func{Re}\left(
q_{1}\right) =0 \\
&=&\left\langle u_{1}q_{1},q_{1}\right\rangle \\
&=&\left\langle \func{Im}\left( u_{1}q_{1}\right) ,q_{1}\right\rangle ,\text{
since }\func{Re}\left( q_{1}\right) =0 \\
&=&\left\langle \func{Im}\left( u_{2}q_{2}\right) ,q_{2}\right\rangle ,\text{
by Equation (\ref{3rd comp eqn}) and the fact that }\phi \left( u_{1}\right)
=\phi \left( u_{2}\right) \\
&=&\func{Re}\left( u_{2}\right)
\end{eqnarray*}

and%
\begin{eqnarray*}
\left\langle \func{Im}\left( u_{1}\right) ,q_{1}\right\rangle
&=&-\left\langle u_{1},\bar{q}_{1}\right\rangle ,\text{ since }\func{Re}%
\left( q_{1}\right) =0 \\
&=&-\func{Re}\,u_{1}q_{1} \\
&=&-\func{Re}\,u_{2}q_{2},\text{ by Equation \ref{2nd coord eqn} and the
fact that }\phi \left( u_{1}\right) =\phi \left( u_{2}\right) \\
&=&-\left\langle u_{2},\bar{q}_{2}\right\rangle \\
&=&\left\langle \func{Im}\left( u_{2}\right) ,q_{2}\right\rangle .
\end{eqnarray*}%
Together with $\left\vert u_{1}\right\vert =\left\vert u_{2}\right\vert $
and the fact that $q_{1}$ and $q_{2}$ are imaginary, the previous two
displays imply that $\left( 
\begin{array}{c}
u_{1} \\ 
q_{1}%
\end{array}%
\right) $ and $\left( 
\begin{array}{c}
u_{2} \\ 
q_{2}%
\end{array}%
\right) $ are in the same orbit.

Finally suppose that 
\begin{equation*}
Q_{k}|_{U_{2}}\left( 0,r_{1}\right) =Q_{k}|_{U_{2}}\left( 0,r_{2}\right) .
\end{equation*}%
Then%
\begin{equation*}
\left( 1,\func{Re}\left( r_{1}\right) ,0\right) =\left( 1,\func{Re}\left(
r_{2}\right) ,0\right) .
\end{equation*}%
Since we also have that $\left\vert r_{1}\right\vert =\left\vert
r_{2}\right\vert =1,$ it follows that $\left( 0,r_{1}\right) $ and $\left(
0,r_{2}\right) $ are in the same $G^{\Lambda }$--orbit.
\end{proof}

\begin{keylemma}
\label{key lemme}Let $Q_{s}$ be as in Lemma \ref{Q_s Lemma}.

\noindent 1. There is a well-defined surjective map%
\begin{equation*}
\tilde{Q}_{k}:S_{k}^{2b-2}\rightarrow \mathbb{S}^{2b-2}/G^{\Lambda }
\end{equation*}

whose fibers coincide with the orbits of the $G^{\Lambda }$\ action on $%
S_{k}^{2b-2}.$

\noindent 2. The orbit types of $p\in S_{k}^{2b-2}$ and $Q_{s}^{-1}\left( 
\tilde{Q}_{k}\left( p\right) \right) $\ coincide.

\noindent 3. For $p\in S_{k}^{2b-2}$\ and any $q\in $\ $Q_{s}^{-1}\left( 
\tilde{Q}_{k}\left( p\right) \right) $\ the isotropy representation of $%
G_{p}^{\Lambda }$\ and $G_{q}^{\Lambda }$\ are equivalent.

In particular, $\mathbb{S}^{2b-2}/G^{\Lambda }$\ and $S_{k}^{2b-2}/G^{%
\Lambda }$\ are equivalent orbit spaces.
\end{keylemma}

\begin{proof}
Motivated by \cite{GromMey,Wilh1}, we let $h_{1},h_{2}:\Lambda \times 
\mathbb{S}^{b-2}\rightarrow \mathbb{S}^{2b-2}$ be given by 
\begin{eqnarray*}
h_{1}\left( u,q\right) &=&\left( 
\begin{array}{c}
uq \\ 
q%
\end{array}%
\right) \phi \left( u\right) \text{ and} \\
h_{2}\left( v,r\right) &=&\left( 
\begin{array}{c}
r \\ 
\bar{v}r%
\end{array}%
\right) \phi (v).
\end{eqnarray*}

We claim that $Q_{s}$ and $Q_{k}$ are related by%
\begin{equation}
Q_{k}=%
\begin{cases}
Q_{s}\circ h_{1} & \text{on }U_{1} \\ 
Q_{s}\circ h_{2} & \text{on }U_{2}%
\end{cases}%
.  \label{Q_e vs Q_s eqn}
\end{equation}%
Indeed, 
\begin{eqnarray}
Q_{s}\circ h_{1}\left( u,q\right) &=&Q_{s}\left( 
\begin{array}{c}
uq \\ 
q%
\end{array}%
\right) \phi \left( u\right)  \notag \\
&=&\phi (u)\left( |u|,\func{Re}\,uq,\text{ }\phi (u)\left\langle \func{Im}%
\,uq,\func{Im}\,q\right\rangle \right) \\
&=&Q_{k}\left( u,q\right)  \notag
\end{eqnarray}%
and 
\begin{eqnarray*}
Q_{s}\circ h_{2}\left( v,r\right) &=&Q_{s}\left( 
\begin{array}{c}
r \\ 
\bar{v}r%
\end{array}%
\right) \phi (v) \\
&=&\phi (v)\left( \left\vert r\right\vert ,\func{Re}\left( r\right) ,\phi
(v)\left\langle \func{Im}\,r,\func{Im}\bar{v}r\right\rangle \right) \\
&=&Q_{k}\left( v,r\right) ,
\end{eqnarray*}%
proving Equation (\ref{Q_e vs Q_s eqn}).

Since $h_{1}\left( \Lambda \times \mathbb{S}^{b-2}\right) \cup h_{2}\left(
\Lambda \times \mathbb{S}^{b-2}\right) =\mathbb{S}^{2b-2},$ Equation (\ref%
{Q_e vs Q_s eqn}) implies that $Q_{k}\left( S_{k}^{2b-2}\right) =Q_{s}\left( 
\mathbb{S}^{2b-2}\right) ;$ so setting $\tilde{Q}_{k}=Q_{k}$ gives a
well-defined surjective map%
\begin{equation*}
\tilde{Q}_{k}:S_{k}^{2b-2}\rightarrow \mathbb{S}^{2b-2}/G^{\Lambda },
\end{equation*}%
and Part 1 is proven. Parts 2 and 3 follow from the observation that $h_{1}$
and $h_{2}$ are $G^{\Lambda }$--equivariant embeddings.
\end{proof}

Since the antipodal map $A:\mathbb{S}^{2b-2}\longrightarrow \mathbb{S}%
^{2b-2} $ and the involution 
\begin{equation*}
T:S_{k}^{2b-2}\longrightarrow S_{k}^{2b-2}
\end{equation*}%
from page \pageref{dfn of T page}, commute with the $G^{\Lambda }$--actions (%
\ref{quat lin act}), (\ref{oct lin act}) and (\ref{Davis}), they induce
well-defined $\mathbb{Z}_{2}$--actions on our orbit space $Q_{s}\left( 
\mathbb{S}^{2b-2}\right) =Q_{e}\left( S_{e}^{2b-2}\right) =$ \ 
\begin{equation*}
\left\{ \left. \left( x,y,z\right) \text{ }\right\vert \text{ }x\in \left[
0,1\right] ,\text{ }y\in \left[ -x,x\right] ,\text{ }z\in \left[ -\sqrt{%
\left( x^{2}-y^{2}\right) \left( 1-x^{2}\right) },\sqrt{\left(
x^{2}-y^{2}\right) \left( 1-x^{2}\right) }\right] \right\} .
\end{equation*}%
A simple calculation shows that the two $\mathbb{Z}_{2}$--actions on $%
Q_{s}\left( \mathbb{S}^{2b-2}\right) $ coincide and are given by%
\begin{equation*}
\left( x,y,z\right) \mapsto \left( x,-y,z\right) .
\end{equation*}%
Since quotient maps of isometric group actions preserve lower curvature
bounds , $\mathbb{S}^{2b-2}/\left( SO\left( 3\right) \times \mathbb{Z}%
_{2}\right) $ has curvature $\geq 1$ (\cite{BGP}). Therefore, Theorem \ref%
{main thm} follows from Theorem \ref{lifting thm} and Key Lemma \ref{key
lemme}.

\section{Some Closing Remarks}

In the same paper, Hirsch and Milnor also constructed exotic $\mathbb{R}%
P^{5} $s, $P_{k}^{5}$s$.$ The Davis action also descends to the $P_{k}^{5}$s
where they commute with an $SO\left( 2\right) $--action. The combined $%
SO\left( 2\right) \times SO\left( 3\right) $--action on the $P_{k}^{5}$s is
by cohomogeneity one$.$ Dearricott and Grove--Ziller observed that since
these cohomogeneity one actions have codimension 2 singular orbits, Theorem
E of \cite{GrovZil1} implies that they admit invariant metrics of
nonnegative curvature.

Octonionically, the Hirsch-Milnor construction yields closed $13$%
--manifolds, $P_{k}^{13},$ that are homotopy equivalent to $\mathbb{R}%
P^{13}. $ Their proof that the $P_{k}^{5}$s are not diffeomorphic to $%
\mathbb{R}P^{5} $ breaks down, since in contrast to dimension $6,$ there is
an exotic $14$--sphere; however, Chenxu He has informed us that some of the $%
P_{k}^{13}$s are in fact exotic (\cite{He2}).

The Davis construction yields a cohomogeneity one action of $SO\left(
2\right) \times G_{2}$ on the $P_{k}^{13}$s, only now one of the singular
orbits has codimension 6. So we cannot apply Theorem E of \cite{GrovZil1}.
Moreover, there are cohomogeneity one manifolds that do not admit invariant
metrics with nonnegative curvature (\cite{GVWZ,He}). On the other hand, by
the main theorem of \cite{SchTu2}, every cohomogeneity one manifold admits
an invariant metric with almost nonnegative curvature.

\end{document}